\documentclass[12pt,a4paper]{amsart}

\usepackage{amssymb}
\usepackage{mathabx}
\usepackage{bbm}
\usepackage{amsthm}
\usepackage{dsfont}
\usepackage{amsmath}
\usepackage{color,graphics,srcltx}
\usepackage{bm}
\usepackage{easybmat}
\usepackage{multirow,bigdelim}
\usepackage{enumerate}
\usepackage{graphicx}
\usepackage{amsthm}
\usepackage{url}
\usepackage{tikz-cd}
\usepackage{hyperref}
\hypersetup{colorlinks=true,allcolors=blue}
\usepackage{orcidlink}

\newtheorem{proposition}{Proposition}
\newtheorem{lemma}[proposition]{Lemma}
\newtheorem{corollary}[proposition]{Corollary}
\newtheorem{theorem}[proposition]{Theorem}

\theoremstyle{definition}
\newtheorem{example}[proposition]{Example}
\newtheorem{definition}[proposition]{Definition}

\newcommand{\F}{{\mathbb F}}
\newcommand{\Q}{\mathbb{Q}}
\newcommand{\Z}{\mathbb{Z}}
\DeclareMathOperator{\ch}{char}
\newcommand{\rng}[1]{\langle #1 \rangle}
\def\id{\mathrm{id}}

\begin{document}

\title{When does an infinite ring have a finite compressed commuting graph?}

\author[I.-V. Boroja]{Ivan-Vanja Boroja \orcidlink{0000-0003-1836-0381}}
\address{University of Banja Luka, Faculty of Electrical Engineering, Bosnia and Herzegovina}
\email{ivan-vanja.boroja@etf.unibl.org}

\author[D. Kokol Bukov\v sek]{Damjana Kokol Bukov\v{s}ek \orcidlink{0000-0002-0098-6784}}
\address{University of Ljubljana, School of Economics and Business, and Institute of Mathematics, Physics and Mechanics, Ljubljana, Slovenia}
\email{damjana.kokol.bukovsek@ef.uni-lj.si}

\author[N. Stopar]{Nik Stopar \orcidlink{0000-0002-0004-4957}}
\address{University of Ljubljana, Faculty of Civil and Geodetic Engineering, University of Ljublja\-na, Faculty of Mathematics and Physics, and Institute of Mathematics, Physics and Mechanics, Ljubljana, Slovenia}
\email{nik.stopar@fgg.uni-lj.si}

\begin{abstract}
We show that any infinite ring has an infinite nonunital compressed commuting graph. We classify all infinite unital rings with finite unital compressed commuting graph, using semidirect product of rings as our main tool.
As a consequence we also classify infinite unital rings with only finitely many unital subrings.
\end{abstract}


\maketitle

\section{Introduction}\label{sec:intro}

In recent years there has been a lot of interest in the study of graphs of algebraic structures, constructed from relations on these structures. Such examples are zero-divisor graph \cite{AnLi}, total graph \cite{AnBa}, power graph \cite{AbKeCh}, etc. One of the most studied graphs of algebraic structures is the commuting graph. This is a simple graph whose vertices are all non-central elements of a structure $A$ equipped with a product operation (e.g. ring, group, semigroup, etc.) and where two distinct vertices $a, b$ are connected if they commute in $A$, i.e., if $ab = ba$. The commuting graph was first introduced for groups in \cite{BrFo} in an early attempt towards classification of simple finite groups. It was later extended to rings in \cite{AkGhHaMo04} and other algebraic structures \cite{ArKiKo, WaXi}. Commuting graphs have seen a lot of attention in the last two decades, see for example \cite{AkBiMo, DoD19, DoKoKu18, DoMa24, Ku18}. 

In the paper \cite{BoDoKoSt23} a (unital) compressed commuting graph of a (unital) ring was introduced in order to make the graph smaller,
while still keeping the essence of the commutativity relation. 
The idea of compression is to combine certain vertices of the graph into a single vertex. For this to work, only vertices that are indistinguishable in the non-compressed graph can be combined. 
The compression is defined in such a way that the vertices of compressed commuting graph are in a bijective correspondence with subrings generated by one element. 
This means that the compressed commuting graph
takes into account not only the commuting structure of the ring in question but also the commuting structure of all possible homomorphic images of that ring. See Section~\ref{sec:prelim} for more details.

Since the the aim of compression is to make the graph smaller, it is natural to ask when does an infinite ring have a finite compressed graph.   
For the zero-divisor graph this question was recently considered and answered in \cite{SbZa23}. 
The aim of this paper is to answer the question for the commuting graph, both in the unital and nonunital setting. 
We show that an infinite ring always has an infinite nonunital compressed commuting graph. 
On the other hand, a unital compressed commuting graph of an infinite unital ring can be finite, namely, this happens if and only if the ring is a semidirect product of specific rings. 
The above question is closely related to the question whether an infinite (unital) ring can have only finitely many (unital) subrings generated by one element. Thus, most of our results, although obtained in the framework of graphs, can be understood as purely algebraic results and might therefore be of interest to a more general audience. 

The paper is structured as follows. In Section~\ref{sec:prelim} we recall the definition of (unital) compressed commuting graph and its basic properties. 
In Section~\ref{sec:nonunital} we answer the question discussed above in the nonunital setting. 
Section~\ref{sec:semidirect} is devoted to semidirect product of rings, a construction  crucial for our results. 
The answer to the question in the unital setting is given in Section~\ref{sec:unital}, and in Section~\ref{sec:fields} we apply our results to fields. 
Our findings are illustrated  with several interesting examples.

\section{Preliminaries} \label{sec:prelim}

In the paper \cite{BoDoKoSt23} (unital) compressed commuting graph of a ring is introduced. Here we recall the definitions and their basic properties. 

Let $R$ be a general ring, possibly nonunital. A subring of $R$ generated by an element $a \in R$ will be denoted by $\rng{a}$, i.e., $\rng{a}=\{q(a) ~|~ q \in \Z[x],\, q(0)=0\}$, where $\Z[x]$ denotes the ring of polynomials with integer coefficients.
An equivalence relation $\sim$ on $R$ is defined by $a \sim b$ if and only if $\rng{a}=\rng{b}$, and the equivalence class of an element $a \in R$ with respect to relation $\sim$ is denoted by $[a]$.
The class $[a]$ consists of all single generators of the ring $\rng{a}$.

\begin{definition}\label{def:CCG}
A \emph{compressed commuting graph} of a ring $R$ is an undirected graph $\Lambda(R)$ whose vertex set is the set of all equivalence classes of elements of $R$ with respect to relation $\sim$ and there is an edge between $[a]$ and $[b]$ if and only if $ab=ba$.
\end{definition}

Note that edges in $\Lambda(R)$ are well defined. Central elements of $R$ are not excluded from the graph $\Lambda(R)$ like in the usual commuting graph. Furthermore, loops are allowed in $\Lambda(R)$, in fact, every vertex of $\Lambda(R)$ has a loop.

The mapping $\Lambda$ can be extended to a functor $\Lambda$ from the category $\mathbf{Ring}$ of (possibly nonunital) rings and ring  homomorphisms to the category $\mathbf{Graph}$ of undirected simple graphs that allow loops and graph morphisms. 
For a ring homomorphism $f \colon R \to S$ a graph morphism $\Lambda(f) \colon \Lambda(R) \to \Lambda(S)$ is defined by $\Lambda(f)([r])=[f(r)]$.
The mapping $\Lambda \colon \mathbf{Ring} \to \mathbf{Graph}$ that maps a ring $R$ to the graph $\Lambda(R)$ and a ring morphism $f$ to the graph morphism $\Lambda(f)$ is a functor which preserves embeddings.

A unital version of the above functor for unital rings can be defined as follows.
Let $R$ be a unital ring with identity element $1$.
Let $\rng{a}_1$ denote the subring of $R$ generated by $1$ and $a \in R$. The ring $\rng{a}_1$ will be refered to as the \emph{unital subring} of $R$ generated by $a$.
Note that $\rng{a}_1=\{q(a) ~|~ q \in \Z[x]\}$, where in the evaluation $q(a)$ the constant term of $q$ is multiplied by the identity element $1 \in R$.
An equivalence relation on $R$ is defined by $a \sim_1 b$ if and only if $\rng{a}_1=\rng{b}_1$ and the equivalence class of an element $a \in R$ is denoted by $[a]_1$.

\begin{definition}\label{def:UCCG}
    A \emph{unital compressed commuting graph} of a unital ring $R$ is an undirected graph $\Lambda^1(R)$ whose vertex set is the set of all equivalence classes of elements of $R$ with respect to relation $\sim_1$ and there is an edge between $[a]_1$ and $[b]_1$ of and only if $ab=ba$.
\end{definition}

Similarly as before the mapping $\Lambda^1$ can be extended to a functor $\Lambda^1$ from the category $\mathbf{Ring^1}$ of unital rings and unital ring morphisms to the category $\mathbf{Graph}$.
Here the zero ring $R=0$ is considered as a unital ring with $1=0$.
For any unital ring homomorphism $f \colon R \to S$, where $R$ and $S$ are unital rings, $\Lambda^1(f) \colon \Lambda^1(R) \to \Lambda^1(S)$ is defined by $\Lambda^1(f)([r]_1)=[f(r)]_1$.
The mapping $\Lambda^1 \colon \mathbf{Ring^1} \to \mathbf{Graph}$ that maps a unital ring $R$ to the graph $\Lambda^1(R)$ and a unital ring morphism $f$ to the graph morphism $\Lambda^1(f)$ is a functor which preserves embeddings.

For a general ring $R$, let $m$ be the least positive integer, if it exists, such that $mR=0$. Otherwise let $m=0$. The number $m$ is called the \emph{characteristic} of ring $R$ and denoted by $\ch R$. An identity element can be adjoined to $R$ to form a unital ring $R^1$ as follows.
Note that $R$ is a left $\Z_m$-module, where $\Z_0=\Z$.
Equip the set $R^1=\Z_m \times R$ with componentwise addition and define multiplication in $R^1$ by $(k,a)(n,b)=(kn,na+kb+ab)$.
Then $R^1$ is a unital ring with identity element $(1,0)$.
If $m=p$ is prime then $R^1$ is also an algebra over $\Z_p \cong GF(p)$ with scalar multiplication defined componentwise.
The ring $R$ is an ideal of $R^1$ with the canonical embedding $i \colon R \to R^1$ given by $i(r)=(0,r)$.
The following connection between functors $\Lambda$ and $\Lambda^1$ holds.

\begin{proposition}\label{prop:iso}
    For any ring $R$ (unital or nonunital) we have $\Lambda(R) \cong \Lambda^1(R^1)$, where the isomorphism $\hat{\imath} \colon \Lambda(R) \to \Lambda^1(R^1)$ is given by $\hat{\imath}([a])=[i(a)]_1$.
\end{proposition}

If $G$ and $H$ are two graphs we denote by $G \vee H$ their join, i.e. the graph with  $V(G \vee H) = V(G) \cup V(H)$ and $E(G \vee H) = E(G) \cup E(H) \cup \{\{a,b\} ~|~ a \in V(G), b \in V(H)\}$, and by $tG$ a disjoint union of $t$ copies of $G$. We also denote by $K_n$ the complete graph on $n$ vertices without any loops and by $K^\circ_n$ the complete graph on $n$ vertices with all the loops.

\section{Nonunital compressed commuting graph}\label{sec:nonunital}

In this section we will prove that an infinite ring always has infinite nonunital compressed commuting graph. In order to prove this, we first need an auxiliary lemma about integral elements in unital rings with nonzero characteristic, which may be of independent interest. 

Recall that an element $a$ of a unital ring $R$ is called \emph{integral} over $\Z$ if there exists a monic polynomial $q \in \Z[x]$ such that $q(a)=0$.
For a polynomial $q \in \Z[x]$ denote by $\delta(q)$ the greatest common divisor of its coefficients.

\begin{lemma}\label{lem:integral}
    Let $R$ be a unital ring with $\ch R \ne 0$ and let $a \in R$. If there exists a polynomial $q\in \Z[x]$ such that $\delta(q)=1$ and $q(a)=0$, then $a$ is integral over $\Z$.
\end{lemma}

\begin{proof}
    Let $m = \ch R$. If $m=1$, then $R$ is the zero ring and there is nothing to prove.
    So suppose $m >1$.
    
    We first prove the lemma in the case $m=p^n$, where $p$ is a prime and $n$ is a positive integer. 
    Write $q$ in the form $q(x)=p s_1(x)+s_0(x)$ where the coefficients of $s_0$ are relatively prime to $p$. So
    \begin{equation}\label{eq:q}
        p s_1(a)+s_0(a)=0.
    \end{equation}
    Since $\delta(q)=1$, the polynomial $s_0$ is nonzero.
    Furthermore, there exists an integer $k$ such that $ks_0$ has leading coefficient equal to $1$ modulo $p^n=m$.
    Since $ma=0$, we may thus assume that $s_0$ has leading coefficient equal to $1$ (otherwise just multiply $q$ by $k$ and reduce its coefficients modulo $m$).
    Let $r_0(x)=0$ and inductively define polynomials $h_i,r_i \in \Z[x]$ for $i=1,2,\ldots,n-1$ as the quotient and remainder in the division algorithm when dividing polynomial $s_1(x)$ by the polynomial $s_0(x)+pr_{i-1}(x)$, i.e.
    \begin{equation}\label{eq:div_alg}
    s_1(x)=h_i(x)\big(s_0(x)+pr_{i-1}(x)\big)+r_i(x),
    \end{equation}
    where $\deg r_i< \deg (s_0 + pr_{i-1}) = \deg s_0$.
    Note that the division can be performed since $\deg r_{i-1} < \deg s_0$, so the leading coefficient of the polynomial $s_0(x)+pr_{i-1}(x)$ is equal to~$1$.
    Next we use induction to prove the following.
    
    \medskip
    {\bf Claim 1.} $p^{n-i}s_0(a)+p^{n-i+1}r_{i-1}(a)=0$ for all $i=1,2,\ldots,n$.
    
    \medskip
    When $i=1$ we need to prove that $p^{n-1}s_0(a)=0$. If we multiply equality \eqref{eq:q} by $p^{n-1}$ we obtain $m s_1(a)+p^{n-1}s_0(a)=0$, and since $m \cdot 1=0$, this implies $p^{n-1}s_0(a)=0$.
    Suppose the equality in Claim 1 holds for some $i < n$.
    Multiplying equality \eqref{eq:div_alg} by $p^{n-i}$ and inserting $x=a$ we obtain
    $$p^{n-i}s_1(a)=h_i(a)\big(p^{n-i}s_0(a)+p^{n-i+1}r_{i-1}(a)\big)+p^{n-i}r_i(a)=p^{n-i}r_i(a),$$
    where the last equality follows by induction hypothesis.
    Hence,
    $$p^{n-i-1}s_0(a)+p^{n-i}r_i(a)=p^{n-i-1}s_0(a)+p^{n-i}s_1(a)=p^{n-i-1}\big(s_0(a)+ps_1(a)\big)=0$$
    by equality \eqref{eq:q}. This proves the induction step and thus Claim 1.

    \medskip
    Taking $i=n$ in Claim 1 we obtain $s_0(a)+pr_{n-1}(a)=0$. Since $s_0$ is monic and $\deg r_{n-1}<\deg s_0$, this proves that $a$ is integral over $\Z$.

    Now assume $m$ is general and write it as $m=p_1^{n_1}p_2^{n_2}\ldots p_k^{n_k}$, where $p_i$ are primes and $n_i$ are positive integers.
    Let $a_i$ be the image of $a$ under the canonical projection $R \to R/p_i^{n_i}R$.
    Clearly $q(a_i)=0$ and the ring $R/p_i^{n_i}R$ has characteristic $p_i^{n_i}$.
    Hence, by the first part of the proof, $a_i$ is integral over $\Z$.
    Let $q_i \in \Z[x]$ be a monic polynomial such that $q_i(a_i)=0$.
    Since the canonical projection $R \to R/p_i^{n_i}R$ maps $q_i(a)$ to $q_i(a_i)=0$, we infer $q_i(a) \in p_i^{n_i}R$, so that $m_iq_i(a)=0$, where $m_i=m/p_i^{n_i}$.
    The integers $m_1,m_2,\ldots,m_k$ are relatively prime, hence, there exist integers $c_i$ such that $c_1m_1+c_2m_2+\ldots+c_km_k=1$.
    Let $d_i=\deg q_i$ and $d=\max_{i=1,2,\ldots,k} d_i$.
    Then the polynomial $s(x)=\sum_{i=1}^k c_i m_i q_i(x) x^{d-d_i}$ has leading term equal to $x^d$, so it is monic and $s(a)=0$. Thus, $a$ is integral over $\Z$.    
\end{proof}

Although we are interested in nonunital compressed commuting graph in this section, our first result is actually about unital compressed commuting graph of an infinite unital ring, since we will need it in our proofs. However, we will later improve this result in Section~\ref{sec:unital}.

\begin{proposition}\label{prop:inf_ring_1}
    If $R$ is an infinite unital ring, then $|V(\Lambda^1(R))|=|R|$, unless $|R|=\aleph_0$,  $\ch R = 0$, and every element of $R$ is annihilated by some nonzero linear polynomial with coefficients in $\Z$.
\end{proposition}

\begin{proof}
    Let $R$ be an infinite unital ring such that $|R| \neq \aleph_0$ or $\ch R \ne 0$ or there exists an element $a \in R$ such that $L(a) \neq 0$ for every nonzero linear polynomial $L \in \Z[x]$.
    Since $|V(\Lambda^1(R))| \leq |R|$, it suffices to prove that $|V(\Lambda^1(R))| \geq |R|$. 
    We consider several cases.

    \medskip
    {\bf Case 1.} $|R|>\aleph_0$.
    For any $a \in R$ we have $[a]_1 \subseteq \rng{a}_1$,
    so that $|[a]_1| \leq |\rng{a}_1| \leq |\Z[x]|=\aleph_0$.
    This implies $|R| \leq |V(\Lambda^1(R))| \cdot \aleph_0 = \max\{|V(\Lambda^1(R))|, \aleph_0\}$ by Axiom of choice. Since $|R|>\aleph_0$, we infer $|V(\Lambda^1(R))| \geq |R|$. 
    
    \medskip
    {\bf Case 2.} $|R|=\aleph_0$, $\ch R \ne 0$, and $R$ contains an element $a$ which is not integral over $\Z$.
    We claim that the chain of subrings
    $$\rng{a}_1 \supseteq \rng{a^2}_1 \supseteq \rng{a^4}_1 \supseteq \rng{a^8}_1 \supseteq \ldots$$
    is strictly decreasing.
    Indeed, if $\rng{a^{2^k}}_1=\rng{a^{2^{k+1}}}_1$, then there exists a polynomial $p \in \Z[x]$ such that $a^{2^k}=p(a^{2^{k+1}})$.
    The polynomial $P(x)=x^{2^k}-p(x^{2^{k+1}})$ satisfies $P(a)=0$ and $\delta(P)=1$, since $x^{2^k}$ is one of its terms.
    Lemma~\ref{lem:integral} implies that $a$ is integral over $\Z$, which contradicts our assumption.
    Hence, the classes $[a]_1,[a^2]_1,[a^4]_1,[a^8]_1,\ldots$ are all different, so that $|V(\Lambda^1(R))| \geq \aleph_0$.
    
    \medskip
    {\bf Case 3.} $|R|=\aleph_0$, $\ch R \ne 0$, and every element of $R$ is integral over $\Z$.
    Let $m= \ch R$.
    By \cite{La72}, $R$ contains an infinite commutative subring $S$.
    Hence, there exists a strictly increasing chain of finite commutative unital subrings of $R$
    $$S_1 \varsubsetneq S_2 \varsubsetneq S_3 \varsubsetneq \ldots \varsubsetneq S.$$
    Indeed, let $S_1=\Z_m$, and for any $k>1$ let $S_k=S_{k-1}[a_k] = \Z_m[a_2,a_3,\ldots,a_k]$ where $a_k \in S \setminus S_{k-1}$.
    These subrings are finite since elements $a_i$ are integral over $\Z$, they commute, and $\Z_m$ is finite. Hence, the chain is infinite since $S$ is infinite.
    It follows that $\Lambda^1(S_k)$ has strictly more vertices than $\Lambda^1(S_{k-1})$, so that $\lim_{k \to \infty} |V(\Lambda^1(S_k))|=\infty$.
    The fact that the functor $\Lambda^1$ preserves embedding implies $|V(\Lambda^1(R))| \geq \aleph_0$.

    \medskip
    {\bf Case 4.} $|R|=\aleph_0$, $\ch R = 0$, and there exists an element $a\in R$ such that $L(a) \neq 0$ for every nonzero linear polynomial $L \in \Z[x]$.
    If $a$ does not satisfy any polynomial equation with coefficients in $\Z$, then we have a strictly decreasing chain of subrings
    $$\rng{a}_1 \varsupsetneq \rng{a^2}_1 \varsupsetneq \rng{a^4}_1 \varsupsetneq \ldots,$$
    which implies $|V(\Lambda^1(R))|\geq \aleph_0$ as in Case 2.
    So we may further assume there is a nonzero polynomial $p \in \Z[x]$ such that $p(a)=0$, and that $p$ has the least degree among all such polynomials. 
    By assumption, $p$ is not linear, so the degree of $p$ must be at least $2$.
    Let the leading term of $p$ be equal to $\alpha x^d$, where $\alpha \in \Z$. Then $\alpha^{d-1}p(x)=P(\alpha x)$ for some monic polynomial $P \in \Z[x]$.
    Hence, $b=\alpha a$ satisfies $P(b)=0$, so it is integral over $\Z$.
    We claim that the chain of subrings
    \begin{equation} \label{eq:chain_2b}
    \rng{b}_1 \supseteq \rng{2b}_1 \supseteq \rng{4b}_1 \supseteq \rng{8b}_1 \supseteq \ldots
    \end{equation}
    is strictly decreasing.
    Suppose on the contrary that $\rng{2^sb}_1=\rng{2^{s+1}b}_1$ for some nonnegative integer $s$.
    Then $2^sb=q(2^{s+1}b)$ for some polynomial $q \in\Z[x]$.
    Note that $q(2^{s+1}x)=2^{s+1}Q(x)+\beta$ for some polynomial $Q \in \Z[x]$ and some $\beta \in \Z$, so that
    $$2^sb=2^{s+1}Q(b)+\beta.$$
    Since polynomial $P$ is monic, we may divide $Q$ by $P$ to obtain
    $Q(x)=h(x)P(x)+r(x)$, where $\deg r<\deg P =\deg p$.
    We then have $Q(b)=r(b)$ and therefore $2^sb=2^{s+1}r(b)+\beta$.
    Recall that $b=\alpha a$, so the last equality implies that the polynomial $t(x)=2^s \alpha x - 2^{s+1}r(\alpha x)-\beta$ annihilates $a$.
    Since its degree is $\deg t \leq \max \{1,\deg r\}<\deg p$, we conclude from the definition of polynomial $p$ that $t$ must be the zero polynomial.
    However, the coefficient of the linear term in $2^{s+1}r(\alpha x)$ is divisible by $2^{s+1}\alpha$ while the coefficient of $2^s \alpha x$ is not, so $t$ cannot be the zero polynomial.
    This contradiction shows that the chain \eqref{eq:chain_2b} is indeed strictly decreasing, and consequently $|V(\Lambda^1(R))|\geq \aleph_0$.
\end{proof}

We can now prove the main theorem of this section.

\begin{theorem}\label{thm:inf_ring}
    If $R$ is an infinite ring, then $|V(\Lambda(R))|=|R|$.
\end{theorem}

\begin{proof}
    Suppose $R$ is an infinite ring. Then $|R^1|=|R|$ by Axiom of choice.
    Propositions~\ref{prop:iso} and \ref{prop:inf_ring_1} imply that $|V(\Lambda(R))|=|V(\Lambda^1(R^1))|=|R^1|=|R|$, unless $|R^1|=\aleph_0$, $\ch R^1 = 0$, and every element of $R^1$ is annihilated by some nonzero linear polynomial with coefficients in $\Z$.
    To finish the proof, let $R$ be a ring such that $R^1$ satisfies all three of the above conditions.
    The first two conditions imply $|R|=\aleph_0$ and $mR \neq 0$ for all positive integers $m$.
    By the last condition, for every $a \in R$, the element $i(a) \in R^1$ is annihilated by some linear polynomial $L(x)=nx+n_0$, with $n,n_0 \in \Z$.
    But $L(i(a))=(n_0,na)$, hence, $na=0$.
    This means that every element of $R$ has finite additive order.
    For each positive integer $k$ let $m_k=p_1^{k}p_2^{k-1}p_3^{k-2}\cdots p_{k-1}^{2}p_k^{1}$, where $p_i$ is the $i$-th consecutive prime number, and denote $R_{m_k}=\{r \in R ~|~ m_kr=0\}$.
    We have an increasing chain
    \begin{equation}\label{eq:chain_Rmk}
    R_{m_1} \subseteq R_{m_2} \subseteq R_{m_3} \subseteq \ldots
    \end{equation}
    of subrings of $R$.
    It follows from the above conditions on $R$ that $R_{m_k} \neq R$ for all $k$, and $\bigcup_{k=1}^\infty R_{m_k} =R$, since every positive integer (the additive order of an element) is a divisor of $m_k$ for some $k$ big enough.
    This implies that the chain \eqref{eq:chain_Rmk} does not stabilize, hence, we conclude that $|V(\Lambda(R))| \geq\lim_{k \to \infty} |V(\Lambda(R_{m_k}))|=\infty$.
\end{proof}

\section{Semidirect product of rings}\label{sec:semidirect}

We now turn our attention to semidirect products of rings, defined bellow, since this will be a crucial construction in the rest of the paper. 

We start by an example of an infinite unital ring with a finite unital compressed commuting graph, which shows that Theorem \ref{thm:inf_ring} cannot be directly generalized to the unital case.

\begin{example} \label{ex:Z1/m}
 Consider the ring $\Z[\frac{1}{m}]$. Theorem \ref{thm:inf_ring} shows that $\Lambda(\Z[\frac{1}{m}])$ is infinite, in fact $\Lambda(\Z[\frac{1}{m}]) = K^{\circ}_{\aleph_0}$ since the ring is commutative.
 On the other hand, the unital compressed commuting graph of $\Z[\frac{1}{m}]$ is finite, namely, as shown below, $\Lambda^1(\Z[\frac{1}{m}]) = K^{\circ}_{2^s}$, where $s$ is the number of prime divisors of $m$. So the unital version of  Theorem \ref{thm:inf_ring} is not true.

 To prove the above claim, let $\{p_1,p_2,\ldots,p_s\}$ be the set of all prime divisors of $m$.
 Every element of $\Z[\frac1m]$ is of the form $a=n/(p_{i_1}^{k_1}p_{i_2}^{k_2}\cdots p_{i_t}^{k_t})$ for some integers $1 \le i_1 <i_2<\ldots<i_t \le s$, some positive integers $k_1,k_2,\ldots,k_t$, and some integer $n$ relatively prime to $p_{i_1}, p_{i_2},\ldots, p_{i_t}$ (here $t=0$ iff the denominator is $1$).
 We claim that $[a]_1=[1/(p_{i_1}p_{i_2}\cdots p_{i_t})]_1$. Denote $q=1/(p_{i_1}p_{i_2}\cdots p_{i_t})$.
 Chose an integer $k$ greater than all of $k_1,k_2,\ldots,k_t$ and note that $a=n(p_1^{k-k_1}p_2^{k-k_2}\cdots p_t^{k-k_t})q^k \in \rng{q}_1$. 
 Since $n$ is relatively prime to $p_{i_1}, p_{i_2},\ldots, p_{i_t}$, there exist integers $u$ and $v$ such that $un+v(p_{i_1}^{k_1}p_{i_2}^{k_2}\cdots p_{i_t}^{k_t})=1$, so that $q=(p_{i_1}^{k_1-1}p_{i_2}^{k_2-1}\cdots p_{i_t}^{k_t-1})(ua+v) \in \rng{a}_1$.
 Hence, $\rng{a}_1=\rng{q}_1$ and $[a]_1=[q]_1$.
 Furthermore, if $q' = 1/(p_{j_1}p_{j_2}\cdots p_{j_{r}})$ for some integers $1 \le j_1 <j_2<\ldots<j_{r} \le s$, then $[q]_1=[q']_1$ if and only if $q=q'$, because the denominator of any reduced fraction in $\rng{q}_1$ can only be divisible by primes in $\{p_{i_1}, p_{i_2},\ldots, p_{i_t}\}$.
 This implies that there are as many vertices in $\Lambda^1(\Z[\frac 1m])$ as there are divisors of $p_1p_2\cdots p_s$, namely, $2^s$.
 \end{example}

Next example shows that the functor $\Lambda^1$ does not behave nicely with respect to direct product of rings. 
 
 \begin{example}
 Consider now a direct product $\Z \times \Z$. Then $\Lambda^1(\Z \times \Z)$ is infinite even though $\Lambda^1(\Z) = K^{\circ}_1$ is finite. 
 Indeed, the element $(1,0) \in \Z \times \Z$ is not a zero of any linear polynomial with coefficients in $\Z$, so by Proposition~\ref{prop:inf_ring_1} the graph $\Lambda^1(\Z \times \Z)$ is infinite. 
\end{example}

We now introduce a \emph{semidirect product} of rings. We mimic the construction that is standard for groups \cite{Ro95} and Lie algebras \cite{Pa07}, but seems to be less standard for rings, at least in the form that we will need here.
Let
$$\begin{tikzcd}[sep=20pt]
& & & Z & \\
0 & I & R & Z & 0
\arrow[from=2-1, to=2-2]
\arrow["f", from=2-2, to=2-3]
\arrow["g", from=2-3, to=2-4]
\arrow[from=2-4, to=2-5]
\arrow["h"', from=1-4, to=2-3]
\arrow["id", from=1-4, to=2-4]
\end{tikzcd}$$
be a commutative diagram of rings and ring morphisms with an exact row. Such a diagram is sometimes called a \emph{split exact sequence}.
Then $I$ can be given a structure of a \emph{two-sided $Z$-ring}, i.e., a $Z$-bimodule which is also a ring and the two structures are connected by
$$z \cdot (x_1x_2)=(z \cdot x_1)x_2, \quad
x_1(z \cdot x_2)=(x_1 \cdot z)x_2, \quad
(x_1x_2) \cdot z=x_1(x_2 \cdot z)$$
for all $z \in Z$ and $x_1,x_2 \in I$.
For any $x \in I$ and $z \in Z$ the scalar multiplications $z\cdot x$ and $x \cdot z$ are defined as the elements of $I$ that satisfy the conditions
$$f(z\cdot x)=h(z)f(x) \qquad \text{and} \qquad f(x\cdot z)=f(x)h(z).$$
It is easy to see that the exactness of the row implies that such elements exist and are unique.
Furthermore, the diagram implies that $f(I)$ and $h(Z)$ are subrings of $R$ with the properties $f(I) \cong I$, $h(Z) \cong Z$, $h(Z)+f(I)=R$ and $f(I) \cap h(Z)=0$.
Note that if $Z$ is a unital ring, its identity element $1_Z$ does not necessarily act as the identity on $I$, neither from the left nor from the right (see also below).

The above allows one to prove that $R$ is isomorphic to the ring $Z \ltimes I=(Z \times I,+,\cdot)$, where the addition is defined componentwise and the multiplication is given by
$$(z_1,x_1)\cdot(z_2,x_2)=(z_1z_2,z_1 \cdot x_2+x_1 \cdot z_2+x_1x_2).$$
We will call $Z \ltimes I$ a \emph{semidirect product} of ring $Z$ and $Z$-ring $I$.
The isomorphism $Z \ltimes I \to R$ is given by $(z,x) \mapsto h(z)+f(x)$.

The semidirect product $Z \ltimes I$ is a unital ring if and only if $Z$ is unital with identity element $1_Z$ and $I$ contains an idempotent $e$ such that the action of $1_Z$ on $I$ is given by
\begin{equation}\label{eq:usp}
    1_Z \cdot x=x-ex \qquad\text{and}\qquad x \cdot 1_Z=x-xe
\end{equation}
for all $x \in I$. In this case the identity element of $Z \ltimes I$ is $(1_Z,e)$.
Indeed, let $(z_0,x_0)$ be the identity element of $Z \ltimes I$. Then
$$(z,0)=(z_0,x_0) \cdot (z,0)=(z_0z,x_0\cdot z)$$
for every $z \in Z$, so that $z_0z=z$. Similarly $zz_0=z$, hence $Z$ is unital with identity element $1_Z=z_0$.
In addition, $x_0 \cdot z=0$ and similarly $z \cdot x_0=0$ for every $z \in Z$. In particular, $x_0 \cdot z_0=0$.
Furthermore,
$$(0,x)=(z_0,x_0)\cdot (0,x)=(0,z_0 \cdot x+x_0x)$$
for all $x \in I$, so that $z_0 \cdot x=x-x_0x$. Similarly, $x \cdot z_0=x-xx_0$.
In particular, $0 = x_0 \cdot z_0 =$ $x_0 - x_0^2$, hence $e=x_0$ is an idempotent.
Conversely, for every $(z,x) \in Z \ltimes I$ we have
\begin{align*}
(1_Z,e) \cdot (z,x) &=(1_Z z,1_Z \cdot x+e\cdot z+ex)=(z,x-ex+e\cdot (1_Z z)+ex)\\
&=(z,x+(e \cdot 1_Z)\cdot z)=(z,x+(e-e^2)\cdot z)=(z,x),
\end{align*}
and similarly $(z,x) \cdot (1_Z,e)= (z,x)$.

Note that the left action of $1_Z$ on $I$ is the projection of $I$ to $(1-e)I=\{x-ex ~|~ x \in I\}$ along $eI$ and the right action of $1_Z$ on $I$ is the projection of $I$ to $I(1-e)=\{x-xe ~|~ x \in I\}$ along $Ie$, i.e., for all $x \in I$ we have
\begin{align*}
   && 1_Z \cdot (ex) &=0, &  (xe) \cdot 1_Z &=0, &&\\
   && 1_Z \cdot (x-ex) &= (x-ex), & (x-xe) \cdot 1_Z &= (x-xe). &&
\end{align*}
If the conditions in \eqref{eq:usp} are satisfied, we will call $Z \ltimes I$ a \emph{unital semidirect product} of unital ring $Z$ and $Z$-ring $I$.
If $R$ is unital and the morphism $g$ is unital, then the isomorphism $Z \ltimes I \to R$ above is also unital, since the identity element is unique.

We remark that, using the notation from Section~\ref{sec:prelim}, we have $R^1 =\Z_m \ltimes R$ with the natural action $\hat{n} \cdot x=x \cdot \hat{n}=x+x+\ldots+x$ ($n$-times) for all $x \in R$ and $\hat{n} \in \Z_m$ with representative $n\in \Z$. 

The main theorem of this section explains why semidirect products are crucial for our considerations. 

\begin{theorem}\label{thm:inf_ring_2}
    If $I$ is a finite $\Z[\frac1m]$-ring, then the unital compressed commuting graph of a unital semidirect product $\Z[\frac1m] \ltimes I$ is finite. In fact, $|V(\Lambda^1(\Z[\frac1m] \ltimes I))| \le |V(\Lambda^1(\Z[\frac1m]))| \cdot |I|$.
\end{theorem}

\begin{proof}
Let $a = n/(p_{i_1}^{k_1}p_{i_2}^{k_2}\cdots p_{i_t}^{k_t}) \in \Z[\frac1m]$ be an arbitrary element, where $p_{i_1}, p_{i_2}, \ldots, p_{i_t}$ are some of the prime divisors of $m$, $k_1, k_2, \ldots, k_t$ are positive integers, and $n$ is relatively prime to $p_{i_1}, p_{i_2}, \ldots, p_{i_t}$. Consider the powers of an element $(a,r) \in \Z[\frac1m] \ltimes I$. We have $(a,r)^u = (a^u, r')$, where $r' \in I$ depends on the choice of $a, u$, and $r$. Consider $r'$ as a function of $r$, i.e., define a function $f_{a,u}\colon I \to I$ so that $(a,r)^u = (a^u, f_{a,u}(r))$. Since $I$ is a finite ring, there exist only finitely many functions from $I$ to $I$. So there exist positive integers $v < u$, such that $f_{a,u}(r) = f_{a,v}(r)$ for all $r \in I$.
Let $H_a(x) = x^u - x^v$. We have shown that 
$$H_a((a,r)) = (a^u, f_{a,u}(r)) - (a^v, f_{a,v}(r)) = (H_a(a),0)$$
for all $r \in I$.

We claim that the classes $[H_a(a)]_1$ and $[a]_1$ are equal in $\Z[\frac1m]$. As proven in Example \ref{ex:Z1/m} we have $[a]_1 = [b]_1$, where $b = 1/(p_{i_1}p_{i_2}\cdots p_{i_t})$.
Furthermore, 
$$H_a(a) = \frac{n^u}{p_{i_1}^{uk_1}p_{i_2}^{uk_2}\cdots p_{i_t}^{uk_t}} - \frac{n^v}{p_{i_1}^{vk_1}p_{i_2}^{vk_2}\cdots p_{i_t}^{vk_t}} =
\frac{n^v(n^{u-v} - p_{i_1}^{(u-v)k_1}p_{i_2}^{(u-v)k_2}\cdots p_{i_t}^{(u-v)k_t})}{p_{i_1}^{uk_1}p_{i_2}^{uk_2}\cdots p_{i_t}^{uk_t}}.$$
Since both $n^v$ and $n^{u-v} - p_{i_1}^{(u-v)k_1}p_{i_2}^{(u-v)k_2}\cdots p_{i_t}^{(u-v)k_t}$ are relatively prime to $p_{i_1}, p_{i_2}, \ldots, p_{i_t}$, we obtain also $[H_a(a)]_1 = [b]_1$.

Let $(1,e)$ be the identity element of the ring $\Z[\frac1m] \ltimes I$. If $P \in \Z[x]$ is a polynomial with constant coefficient $p_0$, then $P((a,0)) = (P(a), p_0e)$, since $(a,0)^2 = (a^2, a \cdot 0 + 0 \cdot a + 0^2) = (a^2, 0)$ and similarly for higher powers. Here $p_0e$ means $\underbrace{e + e + \ldots + e}_{p_0\text{-times}}$.

Since $[a]_1 = [b]_1$, there exist polynomials $P_1, P_2 \in \Z[x]$, such that $P_1(a) = b$ and $P_2(b) = a$. Since $[H_a(a)]_1 = [a]_1$, there exists a polynomial $G_1 \in \Z[x]$, such that $G_1(H_a(a)) = a$. Let $K_1(x) = P_1(x) - x$ and $S_1(x) = K_1(G_1(H_a(x))) + x$. Let $r \in I$ be arbitrary element. We have
\begin{align*}
    S_1((a, r)) &= K_1(G_1(H_a((a, r)))) + (a, r) = K_1(G_1((H_a(a), 0))) + (a, r) \\
    &= (K_1(G_1(H_a(a))), g_1e) + (a, r) = (K_1(a) + a, g_1e + r) = (b, g_1e + r),
\end{align*}
where $g_1 \in \Z$ is the constant coefficient of the polynomial $K_1(G_1(x))$. Note that $g_1$ depends only on $a$ and $b$, but is independent of the choice of $r \in I$. 
Similarly as above, let $G_2 \in \Z[x]$, such that $G_2(H_b(b)) = b$, $K_2(x) = P_2(x) - x$ and $S_2(x) = K_2(G_2(H_b(x))) + x$. As above $S_2((b, r)) = (a, g_2e + r)$ for any $r \in I$, where $g_2$ is the constant coefficient of the polynomial $K_2(G_2(x))$.
It follows that
$(S_2 \circ S_1)(a, r) = (a,g_2e + g_1e + r)$.
Let $c = \ch I$ be the characteristic of $I$. We have
\begin{align*}
&\big(\underbrace{(S_2 \circ S_1) \circ\ldots\circ(S_2 \circ S_1)}_{c\text{-times}}\big)(a, r) = (a, c(g_2+g_1)e + r)  = (a, r), \\
&S_1(a, r) = (b, g_1e + r), \text{ and }\\
&\big(\underbrace{(S_2 \circ S_1) \circ\ldots\circ(S_2 \circ S_1)}_{(c-1)\text{-times}} \circ S_2\big)(b, g_1e + r) = (a, r). 
\end{align*}
It follows that $[(a, r)]_1 = [(b, g_1e + r)]_1$ for all $r \in I$.

We have proven that for arbitrary element $(a, r) \in \Z[\frac1m] \ltimes I$ there exist a typical represen\-tative $b$ of the class $[a]_1$ in $\Z[\frac1m]$ and $r' \in I$ (namely $r' = r + g_1e$), such that $[(a, r)]_1 = [(b, r')]_1$.
This means that the number of classes in $\Z[\frac1m] \ltimes I$ is at most $|V(\Lambda^1(\Z[\frac1m])| \cdot |I|$, so the graph $\Lambda^1(\Z[\frac1m] \ltimes I)$ is finite.
\end{proof}

Given a ring it is not always obvious whether it is isomorphic to some semidirect product. To illustrate Theorem~\ref{thm:inf_ring_2} we give an example of a factor ring that is isomorphic to a direct product $\Z[\frac12] \times \Z_2$ and thus has finite unital compressed commuting graph.

\begin{example} \label{ex:directproduct}
Let $J = \big(2(2x-1), x(2x-1)\big)$ be an ideal in $\Z[x]$ and $R = \Z[x]/J$. We want to show that $R \simeq \Z[\frac12] \times \Z_2$.

A typical member of $R$ is $p(x) + J = a_nx^n + a_{n-1}x^{n-1} + \ldots + a_2x^2 + a_1x + a_0 + J$, where $a_n, \ldots, a_2 \in \{0,1\}, a_1 \in \{0,1,2,3\}$, and $a_0 \in \Z$. 
Indeed, since $2x^2 -x \in J$, we can inductively reduce each coefficient to be 0 or 1, staring by the leading one and ending at $a_2$. Since also $4x -2 \in J$, we can reduce the linear coefficient modulo 4.  Let $\tilde{g}\colon \Z[x] \to \Z[\frac12] \times \Z_2$ be a unital homomorphism defined by $\tilde{g}(p) = (p(\frac12), p(0) \mod 2)$. 
Since $J \subset \mathop{\mathrm{ker}} \tilde{g}$, homomorphism $\tilde{g}$ induces $g\colon R \to \Z[\frac12] \times \Z_2$.
Let $p(x) = a_nx^n + a_{n-1}x^{n-1} + \ldots + a_2x^2 + a_1x + a_0$, where $a_n, \ldots, a_2 \in \{0,1\}, a_1 \in \{0,1,2,3\}$, and $a_0 \in \Z$ and suppose $g(p(x)+J) = (0,0)$. We have 
$$p\left(\frac12\right) = \frac{a}{2^n} + \frac{a_1}{2} + a_0.$$ 
If at least one of the coefficients $a_n, \ldots, a_2$ is non-zero, we may assume that $a_n=1$, so $a\in \Z$ is odd. Since $p(\frac12) = 0$, this is not possible, so $a_n= \ldots =a_2=0$.
Now, $p(\frac12) = \frac{a_1}{2} + a_0 =0$, and since $a_0$ is even, it follows $a_1 =0$ and $a_0=0$, thus homomorphism $g$ is injective.
Write an arbitrary element of $\Z[\frac12] \times \Z_2$ as $(\frac{a}{2^n}, b)$, where $a \in \Z$, $n \ge 1$, and $b \in \{0,1\}$, and define a mapping $h\colon \Z[\frac12] \times \Z_2 \to R$ by $h(\frac{a}{2^n}, b) = ax^n + b(2x-1) +J$.
It is straightforward to check that $g(h(\frac{a}{2^n}, b)) = (\frac{a}{2^n}, b)$, so $g$ is a unital ring isomorphism and $h$ is its inverse.
It follows that $R \simeq \Z[\frac12] \times \Z_2$.

Theorem~\ref{thm:inf_ring_2} now guarantees that the unital compressed commuting graph of $R$ is finite and has at most $|V(\Lambda^1(\Z[\frac12]))| \cdot |\Z_2| = 2 \cdot 2 = 4$ vertices. 
Note that $h(1,0) = 2x+J$ since $1 \in \Z[\frac12]$ is interpreted as $\frac22$. 
The vertices are equivalence classes of elements obtained from vertices of  $\Lambda^1(\Z[\frac12])$ and elements of $\Z_2$, namely $[h(1, 0)]_1 = [2x + J]_1$, $[h(\frac12, 0) + J]_1 = [x + J]_1$, $[h(1, 1)]_1 = [4x - 1 + J]_1 = [1 + J]_1$, and $[h(\frac12, 1)]_1 = [3x -1 + J]_1$. But  $\rng{3x - 1 + J}_1 = \rng{x + J}_1 = R$, since $x + J = 3(3x-1) - 1 + J$. Furthermore, $\rng{1 + J}_1 = \{a+J ~|~ a \in \Z\}$ and $\rng{2x + J}_1 =  \{a+J ~|~ a \in \Z\} \cup  \{2x + a+J ~|~ a \in \Z\}$.  So $\Lambda^1(R)$ actually contains only 3 vertices, and since $R$ is commutative, we have $\Lambda^1(R) = K^\circ_3$, the complete graph on 3 vertices with all the loops. 
\end{example}

\section{Unital compressed commuting graph}\label{sec:unital}

In this section we prove the converse of Theorem~\ref{thm:inf_ring_2} and a version of Theorem~\ref{thm:inf_ring} for unital compressed commuting graphs.

\begin{proposition}\label{thm:inf_ring_1}
    If $R$ is an infinite unital ring, then either $|V(\Lambda^1(R))|=|R|$ or $R$ is isomorphic to a unital semidirect product $\Z[\frac1m] \ltimes I$ for some positive integer $m$ and some finite $\Z[\frac1m]$-ring $I$.
\end{proposition}

\begin{proof}
Assume $R$ is an infinite unital ring such that $|V(\Lambda^1(R))| \neq |R|$.
Proposition~\ref{prop:inf_ring_1} implies that $|R|=\aleph_0$ (hence, $|V(\Lambda^1(R))|<\infty$), $\ch R = 0$, and every element of $R$ is annihilated by some nonzero linear polynomial with coefficients in $\Z$.

The above allows us to define a map $g \colon R \to \Q$ by $g(r)=\frac{a}{b}$ if $r$ is annihilated by the linear polynomial $bx-a \in \Z[x]$.
The fact that $\ch R = 0$ implies $b \neq 0$ for all $r \in R$.
The image $g(r)$ is independent of the choice of the linear polynomial that annihilates $r$. Indeed, if $b_1 x-a_1$ and $b_2x-a_2$ both annihilate $r$, then $b_2a_1=b_1b_2r=b_1a_2$, so $\frac{a_1}{b_1}=\frac{a_2}{b_2}$.
In addition, $g$ is a unital ring morphism, namely, if $r,s \in R$, $bx-a$ annihilates $r$ and $dx-c$ annihilates $s$, then $bd(r+s)=da+bc$ and $bd(rs)=ac$, so $g(r+s)=g(r)+g(s)$ and $g(rs)=g(r)g(s)$. Clearly, $g(1)=1$.
Hence, $g(R)$ is a unital subring of $\Q$.

Clearly, $|V(\Lambda^1(g(R)))| \leq |V(\Lambda^1(R))|<\infty$, which implies that $g(R)$ is finitely generated as a unital ring.
Hence, $g(R)=\Z[\frac{a_1}{b_1},\frac{a_2}{b_2},\ldots,\frac{a_n}{b_n}]$ for some $a_i,b_i \in \Z$, where $a_i$ and $b_i \neq 0$ are relatively prime for each $i$.
It follows that $g(R)=\Z[\frac1m]$, where $m$ is the least common multiple of $b_1,b_2,\ldots,b_n$.
So we have a surjective unital ring morphism $g \colon R \to \Z[\frac1m]$.

Let $I$ be the kernel of $g$ and denote the canonical inclusion of $I$ into $R$ by $f$. Then we have an exact sequence
\begin{equation}\label{eq:exact}
\begin{tikzcd}[sep=20pt]
0 & I & R & \Z[\frac1m] & 0 \ .
\arrow[from=1-1, to=1-2]
\arrow["f", from=1-2, to=1-3]
\arrow["g", from=1-3, to=1-4]
\arrow[from=1-4, to=1-5]
\end{tikzcd}
\end{equation}
Since $\ch R = 0$, we may regard $\Z$ as a subring of $R$.
So $\Z+I$ is a unital subring of $R$. Since $g$ is unital, its restriction to $\Z$ is the identity map so that $\Z \cap I=0$. 

If $k \in \Z$ and $a \in I$, the ring $\rng{k+f(a)}_1$ is generated by $k+f(a)$ and $1$, thus also by $f(a)$ and $1$.
Hence, it contains $\Z\cdot 1 +\rng{f(a)}=\Z + \rng{f(a)} \subseteq \Z+I$. But the latter is clearly a unital subring of $R$, therefore $\rng{k+f(a)}_1=\Z + \rng{f(a)}=\Z\cdot 1 + f(\rng{a})=\rng{f(a)}_1$.
This implies that the map $\hat{f} \colon \Lambda(I) \to \Lambda^1(\Z+I)$ given by $\hat{f}([a])=[f(a)]_1$ is a well defined bijection on the sets of vertices.
Furthermore, for any $a,b \in I$ we have $ab=ba$ if and only if $f(a)f(b)=f(b)f(a)$. Hence, $\hat{f}$ is a graph isomorphism.
This implies that $\Lambda^1(\Z+I) \cong \Lambda(I)$.
Hence, $|V(\Lambda(I))|=|V(\Lambda^1(\Z+I))| \leq |V(\Lambda^1(R))|<\infty$ and the ring $I$ is finite by Theorem~\ref{thm:inf_ring}.

Choose an element $r \in R$ such that $g(r)=\frac1m$. Then $mr-1 \in I$.
Since $(mr-1)^1$, $(mr-1)^2$, $(mr-1)^3,\ldots \in I$ and $I$ is finite, we infer $(mr-1)^u=(mr-1)^v$ for some positive integers $u>v$.
This implies that $mr$ is integral of degree $\leq u$.
In addition, $mr-1$ must have finite additive order, say $n$, so that $n(mr)=n \in \Z$.
Putting the two observations together we infer
$\rng{mr} \subseteq \Z+M$ where
$$M=\Big\{n_1(mr)+n_2(mr)^2+\ldots+n_{u-1}(mr)^{u-1} ~|~ 0 \leq n_1,n_2,\ldots,n_{u-1} <n\Big\}.$$
Since $(mr)^2,(mr)^{2^2},(mr)^{2^3},\ldots \in \rng{mr}$ and the set $M$ is finite we must have $(mr)^{2^j}-(mr)^{2^i} \in \Z$ for some positive integers $j>i$.
On the other hand $(mr)^{2^j}-(mr)^{2^i} \in I$, since $g((mr)^{2^j})=1=g((mr)^{2^i})$, so we conclude that $(mr)^{2^j}-(mr)^{2^i}=0$, because $\Z \cap I=0$.
This implies that the element $e_0=(mr)^k$ with $k=2^j-2^i>0$ is an idempotent in $R$. Indeed,
$$e_0^2=(mr)^{2k}=(mr)^{2^j+(2^j-2\cdot 2^i)}=(mr)^{2^j}(mr)^{2^j-2\cdot 2^i}=(mr)^{2^i}(mr)^{2^j-2\cdot 2^i}=(mr)^{2^j-2^i}=e_0,$$
since $2^j-2\cdot 2^i \geq 0$ (here and in the rest of the proof it should be understood that in a calculation a factor with exponent $0$ is omitted).

Now let $s=r(mr)^{2k-1}$ and define a subring $S=\rng{s}$. We have $g(s)=\frac1m$. Note that $e_0=e_0^2=ms$, hence $e_0 \in S$.
In addition, $s=r(mr)^{2k-1}=(r(mr)^{k-1})e_0=e_0(r(mr)^{k-1})$, so $se_0=s=e_0s$, thus $e_0$ is an identity element in $S$.

We want to prove that $S \cap I=0$. Suppose $a \in S \cap I$.
Then $a=p(s)$ for some polynomial $p \in \Z[x]$ with $p(0)=0$.
Furthermore, applying $g$ to $a$ we obtain $p(\frac1m)=0$.
Denote $d=\deg p$ and define a polynomial $P(x)=x^dp(\frac1x) \in \Z[x]$.
Then $P(m)=0$, so $P(x)=H(x)(x-m)$ for some $H \in \Z[x]$ with $\deg H < \deg P \leq \deg p=d$.
We thus obtain $p(x)=x^dP(\frac1x)=x^{d-1}H(\frac1x)(1-mx)=h(x)(1-mx)$, where $h(x)=x^{d-1}H(\frac1x) \in \Z[x]$.
Furthermore, $h(0)=0$ since $p(0)=0$, so $h(x)=q(x)x$ for some $q \in \Z[x]$ and $p(x)=q(x)(1-mx)x$.
Thus, we get $a=p(s)=q(s)(1-ms)s=q(s)(1-e_0)s=q(s)(s-s)=0$. So $S \cap I=0$.

This implies that $g|_S$, the restriction of $g$ to $S$, is injective. Furthermore, it is also surjective since $g(s)=\frac1m$ and $\Z[\frac1m]$ is generated by $\frac1m$ as a ring, and $g|_S$ is unital since $g(e_0)=g(ms)=mg(s)=1$.
So $g|_S \colon S \to \Z[\frac1m]$ is a unital ring isomorphism.
Define $h \colon \Z[\frac1m] \to R$ by $h(w)=(g|_S)^{-1}(w)$. Then $h$ is a unital ring morphism that extends the exact sequence \eqref{eq:exact} to a commuting diagram
$$\begin{tikzcd}[sep=20pt]
& & & \Z[\frac1m] & \\
0 & I & R & \Z[\frac1m] & 0 \ .
\arrow[from=2-1, to=2-2]
\arrow["f", from=2-2, to=2-3]
\arrow["g", from=2-3, to=2-4]
\arrow[from=2-4, to=2-5]
\arrow["h"', from=1-4, to=2-3]
\arrow["id", from=1-4, to=2-4]
\end{tikzcd}$$
We conclude that $R \cong \Z[\frac1m] \ltimes I$.
\end{proof}

To illustrate the construction in the proof of Proposition~\ref{thm:inf_ring_1} we give an example. 

\begin{example}
Let $J = \big(3(2x-1), (2x-1)^2\big)$ be an ideal in $\Z[x]$ and $R = \Z[x]/J$. We want to show that $R \simeq \Z[\frac12] \ltimes I$, where $I$ is the ring with 3 elements and zero multiplication. 
Here, $\frac12 \in \Z[\frac12]$ acts from both sides on $I$ as $i \mapsto -i$ for every $i\in I$.

A typical member of $R$ is $p(x) + J = a_nx^n + a_{n-1}x^{n-1} + \ldots + a_2x^2 + a_1x + a_0 + J$, where $a_n, \ldots, a_2 \in \{0,1\}, a_1 \in \{0,1,2,3,4,5\}$, and $a_0 \in \Z$. 
Indeed, since $3x(2x-1) - (2x-1)^2 = 2x^2+x-1\in J$, we can inductively reduce each coefficient to be 0 or 1, staring by the leading one and ending at $a_2$. Since also $6x -3 \in J$, we can reduce the linear coefficient modulo 6. Let $\tilde{g}\colon \Z[x] \to \Z[\frac12]$ be a unital homomorphism which maps $x \mapsto \frac12$. Since $J \subset \mathop{\mathrm{ker}} \tilde{g}$, homomorphism $\tilde{g}$ induces $g\colon R \to \Z[\frac12]$. 
Note that for any $p(x) \in \Z[x]$ we have $g(p(x)+J) = p(\frac12)$. Furthermore, the kernel of $g$ equals $\mathop{\mathrm{ker}} g = \{J, 2x-1 + J, 4x-2 + J\}$, which we denote by $I$. 
Indeed, let $p(x) = a_nx^n + a_{n-1}x^{n-1} + \ldots + a_2x^2 + a_1x + a_0$, where $a_n, \ldots, a_2 \in \{0,1\}, a_1 \in \{0,1,2,3,4,5\}$, and $a_0 \in \Z$ and suppose $p(\frac12) = 0$. Similarly as in Example~\ref{ex:directproduct} we obtain $a_n= \ldots =a_2=0$ and $p(\frac12) = \frac{a_1}{2} + a_0 =0$. It follows that $a_1 =0$ and $a_0=0$, or $a_1=2$ and $a_0=-1$, or $a_1=4$ and $a_0=-2$, so $p(x) = 0$ or $p(x) = 2x-1$ or $p(x) = 4x-2$. Note that $I$ is the ring containing 3 elements with usual addition and zero multiplication. 

We now need to find a homomorphism $h\colon \Z[\frac12] \to R$, such that $g \circ h = \id_{\Z[\frac12]}$. 
Following the ideas from the proof of Proposition~\ref{thm:inf_ring_1}, let $r = x + J \in R$. Note that $g(r) = \frac12$. Furthermore, $(2x)^4 + J = 2x + J$, so let $k=3$ and $s = r(2r)^{2k-1} = x(2x)^5 + J = 3x-1 + J$. Let $S = \rng{s} \subseteq R$ and note that $e_0 = (2r)^k = (2x)^3 + J = 1 + J$ is the identity element of $S$. The same proof as for the Proposition~\ref{thm:inf_ring_1} shows that $g|_S \colon S \to \Z[\frac12]$ is a unital ring isomorphism and that $R \simeq \Z[\frac12] \ltimes I$. Let $h$ be the inverse of $g|_S$. 
Since $g(s) = \frac12$, we have $h(\frac12) = s = 3x-1 + J$.
Note that $2(3x-1) + J = 1 + J$. So any element of $S$ can be written as a polynomial in $3x-1 +J$, where all the coefficients except the constant one can be reduced modulo 2. A typical member of $S$ is $p(x) + J = a_n(3x-1)^n + a_{n-1}(3x-1)^{n-1} + \ldots + a_2(3x-1)^2 + a_1(3x-1) + a_0 + J$, where $a_n, \ldots, a_1 \in \{0,1\}$ and $a_0 \in \Z$. Now, $h(\frac12)(2x-1 + J) = (3x-1 + J)(2x-1 + J) = 4x-2 + J$ and $h(\frac12)(4x-2 + J) = (3x-1 + J)(4x-2 + J) = 2x-1 + J$, so $\frac12 \in \Z[\frac12]$ acts from both sides on $I$ as $i \mapsto -i$ for every $i\in I$.

Theorem~\ref{thm:inf_ring_2} now guarantees that unital compressed commuting graph of $R$ is finite and has at most $|V(\Lambda^1(\Z[\frac12]))| \cdot |I| = 2 \cdot 3 = 6$ vertices. The vertices are equivalence classes of elements obtained from vertices of  $\Lambda^1(\Z[\frac12])$ and elements of $I$, namely 
\begin{align*}
    [h(1) + J]_1 &= [1 + J]_1, \\
    \textstyle[h(\frac12) + J]_1 &= [3x - 1 + J]_1, \\
    [h(1) + 2x - 1 + J]_1 &= [2x + J]_1, \\
    \textstyle[h(\frac12) + 2x - 1 + J]_1 &= [5x - 2 + J]_1, \\
    [h(1) + 4x - 2 + J]_1 &= [4x-1 + J]_1, \text{ and} \\
    \textstyle[h(\frac12) + 4x - 2 + J]_1 &= [7x - 3 + J]_1 = [x + J]_1.
\end{align*}
But  $\rng{5x - 2 + J}_1 = \rng{x + J}_1 = R$, since $x + J = 5(5x-2) -2 + J$. Furthermore, 
\begin{align*}
  \rng{1 + J}_1 &= \{a+J ~|~ a \in \Z\}, \ \ \ \rng{3x-1 + J}_1 = S, \text{ and} \\
  \rng{2x + J}_1 &= \rng{4x-1 + J}_1 =  \{a+J ~|~ a \in \Z\} \cup  \{2x + a+J ~|~ a \in \Z\} \cup  \{4x + a+J ~|~ a \in \Z\}.
\end{align*}
So $\Lambda^1(R)$ actually contains only $4$ vertices, and since $R$ is commutative we have $\Lambda^1(R) = K^\circ_4$, the complete graph on $4$ vertices with all the loops. 
\end{example}

In next example we consider an infinite noncommutative ring with finite unital compressed commuting graph, which also shows that the bound in Theorem~\ref{thm:inf_ring_2} is tight.

\begin{example}\label{ex:T_2}
    Let $\mathcal{T}_2(GF(2))$ denote the ring of all $2 \times 2$ upper-triangular matrices over $GF(2)$ and let $R=\Z \ltimes \mathcal{T}_2(GF(2))$, where $1 \in \Z$ acts as the identity on $\mathcal{T}_2(GF(2))$ from both sides.
    First, note that $\Lambda^1(R) \cong \Lambda(\mathcal{T}_2(GF(2)))$ (the argument is the same as the argument in the paragraph after equation~\eqref{eq:exact} in the proof of Proposition~\ref{thm:inf_ring_1}).
    Since $\mathcal{T}_2(GF(2))$ is a subring of $\mathcal{M}_2(GF(2))$, the ring of all $2 \times 2$ matrices over $GF(2)$, its compressed commuting graph $\Lambda(\mathcal{T}_2(GF(2)))$ is a subgraph of $\Lambda(\mathcal{M}_2(GF(2)))$, which was computed in \cite[Theorem~22]{BoDoKoSt23}. The graph $\Lambda(\mathcal{M}_2(GF(2)))$ has $15$ vertices by \cite[Corollary~23]{BoDoKoSt23}, so only two matrices are compressed into a single vertex. These are the matrices with irreducible characteristic polynomial, so they are not upper triangular.
    It follows that $\Lambda^1(R) \cong \Lambda(\mathcal{T}_2(GF(2)))$ has $8$ vertices and is isomorphic to $K_2^\circ \vee (3K_2^\circ)$, see Figure~\ref{fig:upper}.
    Note that $|V(\Lambda^1(R))|=|V(\Lambda^1(\Z))| \cdot |\mathcal{T}_2(GF(2))|$, so the bound in Theorem~\ref{thm:inf_ring_2} is tight.
\end{example}

\begin{figure}
    \centering
    \includegraphics{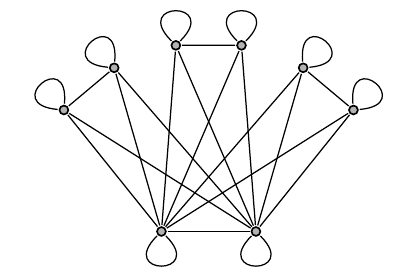}
    \caption{The unital compressed commuting graph of ring $R=\Z \ltimes \mathcal{T}_2(GF(2))$ from Example~\ref{ex:T_2}.}
    \label{fig:upper}
\end{figure}

We collect our findings from Proposition~\ref{thm:inf_ring_1} and Theorem~\ref{thm:inf_ring_2} in the main theorem of this section.

\begin{theorem}\label{thm:inf_ring_3}
    If $R$ is an infinite unital ring, then either $|V(\Lambda^1(R))|=|R|$ or $R$ is isomorphic to a unital semidirect product $\Z[\frac1m] \ltimes I$ for some positive integer $m$ and some finite $\Z[\frac1m]$-ring $I$. In the later case we have $|R|=\aleph_0$ and $|V(\Lambda^1(R))|<\infty$.
\end{theorem}

Since every vertex in the (unital) compressed commuting graph of a (unital) ring $R$ corresponds to some (unital) subring of $R$ generated by one element and vice versa, Theorems~\ref{thm:inf_ring} and \ref{thm:inf_ring_3} imply the following.

\begin{corollary}
    Every infinite ring has infinitely many subrings generated by one element.
\end{corollary}

\begin{corollary}
    An infinite unital ring has only finitely many unital subrings generated by one element if and only if it is isomorphic to a unital semidirect product $\Z[\frac 1m] \ltimes I$ for some positive integer $m$ and some finite $\Z[\frac 1m]$-ring $I$.
\end{corollary}

Furthermore, we also obtain the following.

\begin{corollary}
 Let $R$ be a unital semidirect product $\Z[\frac 1m] \ltimes I$ for some positive integer $m$ and finite $\Z[\frac 1m]$-ring $I$.
 Then $R$ has finitely many unital subrings. 
\end{corollary}

\begin{proof}
    Let $S$ be an unital subring of $R$, generated by $s_1, s_2, \ldots s_n, \ldots$. If $\rng{s_i}_1 = \rng{s_j}_1$, then one of $s_i, s_j$ can be omitted from the set of generators of $S$. So $S$ is finally generated and the minimal number of generators of $S$ is at most the number of unital subrings of $R$ generated by one element. It follows that the number of unital subrings of $R$ is at most $2^n$ where $n = |V(\Lambda^1(R))|$ and is thus finite.
\end{proof}

\section{Infinite fields}\label{sec:fields}

In this section we consider (unital) compressed commuting graph of a field. Let $p$ be a prime and $n \ge 1$ an integer. We denote by $GF(p^n)$ the finite field with $p^n$ elements and by $d(n)$ the number of divisors of $n$. We recall the result about finite fields from \cite{BoDoKoSt23}.

\begin{theorem}\label{th:field}
Let $p$ be a prime and $n \ge 1$ an integer. Then the (unital) compressed commuting graph of the field $GF(p^n)$ is a complete graph on $d(n)$ respectively $d(n)+1$
vertices with all the loops, i.e.,
$$ \Lambda^1(GF(p^n)) \cong K_{d(n)}^\circ \quad\text{ and }\quad \Lambda(GF(p^n)) \cong K_{d(n)+1}^\circ. $$
\end{theorem}

We can now generalize this result to infinite fields. 

\begin{theorem}\label{thm:inf_field}
    If $\F$ is an infinite field, then $\Lambda(\F) \cong \Lambda^1(\F) \cong K_{|\F|}^\circ$.
\end{theorem}

\begin{proof}
    Since $\F$ is commutative its (unital) compressed commuting graph is a complete graph with all the loops.
    Thus, we only need to prove that $|V(\Lambda(F))|=|V(\Lambda^1(\F))|=|\F|$.
    This follows from Theorems~\ref{thm:inf_ring} and \ref{thm:inf_ring_1}, since $\Z[\frac1m] \ltimes I$ cannot be a field.
    Indeed, $I$ would need to be $0$ since it is a proper ideal of $\Z[\frac1m] \ltimes I$, and $\Z[\frac1m]$ would need to be a field, which it is not.
\end{proof}

Next corollary shows that every complete graph with all the loops is realizable as a (unital) compressed commuting graph. 

\begin{corollary}
For any cardinal number $\alpha \ge 1$  there exist a ring $R$ such that its compressed commuting graph is a complete graph on $\alpha$ vertices $K_\alpha^\circ$ with all the loops. 
\end{corollary}

\begin{proof}
If $\alpha$ is finite, let $n= 2^{\alpha-1}$  and $R=GF(p^n)$. Then $d(n) = \alpha$, so $\Lambda^1(R) \cong K_\alpha^\circ$ and $\Lambda(R) \cong K_{ \alpha+1}^\circ$ by Theorem \ref{th:field}. Furthermore, $\Lambda^1(\{0\}) \cong \Lambda(\{0\}) \cong K_1^\circ$. If $\alpha$ is infinite, let $X$ be a set with $|X| = \alpha$, $F$ a finite field and $\F = F(X)$ the field of fractions of the polynomial ring $F[X]$. Since $X$ is infinite, we have $|F(X)| = |X|= \alpha$ by Axiom of choice. Then $\Lambda(\F) \cong \Lambda^1(\F) \cong K_{|\F|}^\circ$ by Theorem \ref{thm:inf_field}.
\end{proof}

\section*{Acknowledgments}
The support by the bilateral grant BI-BA/24-25-024 of the ARIS (Slovenian Research and Innovation Agency) is gratefully acknowledged. Damjana Kokol Bukovšek and Nik Stopar acknowledge financial support from the ARIS (research core funding No. P1-0222 and project J1-50002). Ivan-Vanja Boroja acknowledges financial support from the Municipality of Mrkonjić Grad.

\bibliographystyle{amsplain}
\bibliography{biblio}

\end{document}